\documentclass[a4paper,10pt]{amsart}

\usepackage[utf8]{inputenc}
\usepackage[french,spanish,english]{babel}

\usepackage{amsfonts,amsthm, amssymb,amsmath}

\usepackage{pgf,tikz}
\usetikzlibrary{arrows,shapes,calc,positioning,matrix,external}
\usepackage{tikz-cd}

\usepackage[cmtip,all]{xy}
\xyoption{arc}

\usepackage{graphicx}
\usepackage{subfigure}

\usepackage{url}
\usepackage{microtype}

\usepackage{hyperref}

\usepackage{xcolor}

\usepackage{caption}

\newtheorem{theorem}{Theorem}

\newtheorem{proposition}{Proposition}
\newtheorem*{conjecture*}{Conjecture}

\newtheorem*{coudene}{Coudène's Theorem}

\theoremstyle{definition}
\newtheorem{definition}{Definition}

\newtheorem{remark}{Remark}
\newtheorem{question}{Question}

\newcommand{\R}{\mathbb{R}}
\newcommand{\C}{\mathbb{C}}
\renewcommand{\H}{\mathbb{H}}
\newcommand{\D}{\mathbb{D}}
\newcommand{\Z}{\mathbb{Z}}
\newcommand{\F}{\mathcal{F}}
\newcommand{\scirc}{{\scriptscriptstyle \circ}}
\newcommand{\bs}{\backslash}

\newcommand{\vf}[1]{\frac{\partial~}{\partial #1}}

\title{Unique ergodicity of the horocycle flow on Riemannnian foliations}
\author[F. Alcalde]{Fernando Alcalde Cuesta}
\address{Instituto de Matem\'aticas, Universidade de Santiago de Compostela, E-15782, Santiago de Compostela, Spain. }
\email{fernando.alcalde@usc.es}

\author[F. Dal'Bo]{ Fran\c{c}oise Dal'Bo}
\address{Institut de Recherche Math\'ematique de Rennes, Universit\'e de  Rennes 1, F-35042 Rennes, France}
\email{francoise.dalbo@univ-rennes1.fr}

\author[M. Mart\'{\i}nez]{Matilde Mart\'{\i}nez}
\address{Instituto de Matem\'atica y Estad\'{\i}stica Rafael Laguardia, Facultad de Ingenier\'{\i}a,
Universidad de la Rep\'ublica, J.Herrera y Reissig 565, C.P.11300 Montevideo, Uruguay.}
\email{matildem@fing.edu.uy}

\author[A. Verjovsky]{Alberto Verjovsky}
\address{Universidad Nacional Aut\'onoma de M\'exico,
Apartado Postal 273, Admon. de correos \#3, C.P. 62251 Cuernavaca,
Morelos, Mexico.}
\email{alberto@matcuer.unam.mx}

\date{}

\begin{document}

\maketitle

\begin{abstract}
A classic result due to H.~Furstenberg is the strict ergodicity of the horocycle flow for a compact hyperbolic surface. Strict ergodicity is unique ergodicity with respect to a measure of full support, and therefore implies minimality. The horocycle flow has been previously studied on minimal foliations by hyperbolic surfaces on closed manifolds, where it is known not to be minimal in general. In this paper, we prove that for the special case of Riemannian foliations, strict ergodicity of the horocycle flow still holds. This in particular proves that this flow is minimal, which establishes a conjecture proposed by S.~Matsumoto. 

The main tool is a theorem due to Y.~Coudène, which he presented as an alternative proof for the surface case. It applies to two  continuous flows defining a measure-preserving action of the affine group of the line on a compact metric space, precisely matching the foliated setting.
In addition, we briefly discuss the application of Coudène's theorem to other kinds of foliations.
\end{abstract}
 
\section{Introduction and motivation}

The dynamics of the geodesic and horocycle flows for hyperbolic surfaces has been extensively studied since the seminal works of E. Hopf and G. Hedlund in the 1930s, both from the topological and measure-theoretical points of view. 
In the 1936 paper \cite{Hedlund}, Hedlund proved the minimality of the horocycle flow of any closed surface of constant negative curvature $S$. If $S$ is the quotient of the hyperbolic plane $\H$ by the action of a  torsion-free discrete subgroup $\Gamma$ of the group of orientation-preserving isometries $PSL(2,\R)$, the horocycle flow on the 
unit tangent bundle $T^1 S = \Gamma\backslash PSL(2,\R)$ is given by the right action of the unipotent group
$$U^+=\left\{ \left(
\begin{array}{cc}
 1&s\\
 0&1
\end{array}
\right):\ s\in\R
\right\} $$
on $PSL(2,\R)$. Minimality means that the $U^+$-orbits are dense in $\Gamma\backslash PSL(2,\R)$. 

The strict ergodicity of the horocycle flow was proved almost 40 years later by H. Furstenberg in \cite{Furstenberg2}. Recall that {\em strict ergodicity} is both minimality and unique ergodicity, and was defined in \cite{Furstenberg1}.  Involving techniques of harmonic analysis,  Furstenberg's proof is based on the duality between the right $U^+$-action on $\Gamma\backslash PSL(2,\R)$ and the action of $\Gamma$ on $E=\R^2-\{0\}/\{\pm Id\}$ deduced from the linear action of $SL(2,\R)$ on $\R^2$ (see \cite{Dal'Bo} for details  about this action). Since then this result has been generalised in different frameworks through different techniques by a number of authors (see for instance \cite{Coudene} and references therein).
\bigskip 

Our motivation is to study the dynamics of the horocycle flow  associated to a minimal foliation $\mathcal{F}$  by hyperbolic surfaces  of a closed manifold $M$ according to \cite{MMV}. Like for hyperbolic surfaces, the geodesic and horocycle flows $g_t$ and $h_s^+$ on the unit tangent bundle $\hat M = T^1\F$ are related by 
\begin{equation} \label{geodesichorocyclic}
g_t \scirc h_s^+ = h_{se^{-t}}^+ \scirc g_t,
\end{equation}
and hence provide a continuous joint action of the affine group
$$
B^+=\left\{ \left(
\begin{array}{cc}
 a& b\\
 0 &a^{-1}
\end{array}
\right):\ a,b\in\R, \ a > 0
\right\}.
$$
In general,
the minimality of the foliation $\mathcal{F}$ does not suffice to 
 obtain the minimality of the horocycle flow $h_s^+$ as in Hedlund's theorem
(see examples in \cite{Alcalde-Dal'Bo} and \cite{MMV}).
This question has been addressed in \cite{Alcalde-Dal'Bo}, \cite{ADMV} and \cite{Matsumoto}, obtaining several results under certain restrictions. More concretely:

\begin{enumerate}
 \item In \cite{Alcalde-Dal'Bo}, the question is answered positively for minimal homogeneous $G$-Lie foliations arising from 
 the quotient of the product of $PSL(2,\R)$ and a Lie group $G$ by the action of a cocompact discrete subgroup $\Gamma$.  This manifold is naturally foliated by the orbits of the right $PSL(2,\R)$-action and identifies with the unitary tangent bundle of the foliation $\F$ induced on the quotient 
$$M = \Gamma \bs (PSL(2,\R)/PSO(2,\R) \times G) \simeq \Gamma \bs (\H \times G)$$
by the trivial fibration $D : \H \times G \to G$.

 \item In \cite{ADMV}, the minimality of the horocycle flow is proved for minimal foliations that have a loop which is nonhomotopic to zero in its leaf and which has trivial holonomy.
 
 \item In \cite{Matsumoto}, S. Matsumoto proves that minimality of $h^+_s$ is equivalent to minimality of the $B^+$-action under a more general assumption that includes, among others, all minimal foliations of codimension one.
\end{enumerate}

\noindent
In many cases, though, the question remains unanswered for a large family of foliations by hyperbolic surfaces. 
Focusing on the Riemannian case, directly related with 1 and 3, the minimality of the horocycle flow has been established by the first two authors in the homogeneous case, and later proved by Matsumoto with the additional assumption that there is a leaf which is not simply connected. Let us recall that a \emph{Riemannian foliation} is characterised by the  existence of a Riemannian metric such that 
any geodesic that is orthogonal to the foliation at some point remains orthogonal at every point.
In \cite{Matsumoto}, Matsumoto has formulated the following conjecture:

\begin{conjecture*}[{\cite[Conjecture 1.5]{Matsumoto}}]
For any minimal Riemannian foliation $\mathcal{F}$ by hyperbolic surfaces, the horocycle flow $h_s^+$ is minimal. 
\end{conjecture*}
\noindent Now this conjecture follows as a consequence of the main result in this paper.

\bigskip 

Like the minimality, the strict ergodicity of the horocycle flow $h_s^+$ is a question that arises naturally in the foliated context, but  still less is known.

As Matsumoto observed in \cite{Matsumotoequidis}, Ratner's classification theorem \cite[Theorem 1]{Ratner} can be used to prove the strict ergodicity 
of right $U^+$-action on the quotient $\Gamma \bs (PSL(2,\R) \times G)$ with respect to the 
$PSL(2,\R)$-invariant measure induced by the left Haar measure on $PSL(2,\R) \times G$. Similarly to this example, Riemannian foliations are always equipped with a transverse volume which is invariant by holonomy. 

In fact, few results are known outside of this example, and essentially they all apply to foliated manifolds constructed by suspension and having a transverse projective structure.
If $\Gamma$ is a cocompact discrete subgroup of $SL(2,\R)$ --or more generally a lattice-- and $\rho : \Gamma \to GL(V)$ is a linear representation as group of linear automorphisms of $V = \R^m$ or $\C^m$, we call \emph{suspension of $\rho$} the foliated vector bundle 
$\Gamma \bs (SL(2,\R)\times V)$ obtained from the diagonal action 
$$
\gamma.(u,w) = (\gamma \cdot u,\rho(\gamma)w),
$$
where $\gamma$ acts on $SL(2,\R)$ by left translation. This also yields a foliated fibre bundle $\Gamma \bs (SL(2,\R)\times \mathbb{P}(V))$ with projective fibre $\mathbb{P}(V)$ defined by the induced projective action. As before, these bundles can be interpreted as unitary tangent bundles of foliations on 
$$\Gamma \bs (SL(2,\R)/SO(2,\R) \times V) \simeq \Gamma \bs (\H \times V) \quad \mbox{and} \quad \Gamma \bs (\H \times \mathbb{P}(V))$$ 
respectively, where $\Gamma$ acts now on $\H$ by deck transformations. From the dynamical study of an irreducible $SL(2,\R)$-action on a real or complex vector bundle preserving an ergodic probability measure on the base, C. Bonatti, A. Eskin and A. Wilkinson 
give a sufficient condition for the unique ergodicity of the horocycle flow on the corresponding projective bundle (see \cite[Theorem 2.3]{Bonatti-Eskin-Wilkinson}).
This extends a former result by C. Bonatti and X. G\'omez-Mont about the unique ergodicity of the horocycle flows associated to the suspension of linear representations $\rho : \Gamma \to GL(3,\C)$ such that $\rho(\Gamma)$ does not admit any invariant measure on $\C P^2$, also valid for any generic representation $\rho : \Gamma \to GL(2,\C)$ (see \cite[Th\'eor\`eme 3]{Bonatti-Gomez-Mont}). 

\bigskip

In this work we are interested in a totally different setting: that of a foliation $\F$ by hyperbolic surfaces on a closed manifold $M$ having a {\em transverse invariant measure}. This is a locally finite measure on a complete transversal which is invariant under holonomy. 
Each transverse invariant measure yields, when combined with the Haar measure on $PSL(2,\R)$, a probability measure on $\hat M=T^1\F$ which is invariant for the foliated horocycle flow. Therefore, this flow might have many ergodic measures with full support: 

\begin{remark} For the foliated horocycle flow, minimality does not imply unique ergodicity.
An example in which this flow is minimal but not uniquely ergodic can be constructed from a volume-preserving diffeomorphism $f : T^2 \to T^2$ which is minimal but not uniquely ergodic, like Furstenberg's example \cite{Furstenberg1}. It suffices to take the suspension of the representation of the fundamental group of a closed hyperbolic surface into the group of diffeomorphisms of $T^2$ sending one generator to $f$ and all the other generator to the identity.
\end{remark}

Now, the question that arises naturally is the following:  
\begin{question}
Let $(M,\F)$ be a minimal foliated manifold by hyperbolic surfaces which has a unique transverse invariant measure. Is the foliated horocycle flow uniquely ergodic?
\end{question}
Like in Furstenberg's proof for skew-products, which applies directly to foliations of codimension one, it seems that ergodicity might imply strict ergodicity under this condition. On the other hand, for some classes of foliations it is known that there exists a unique transverse invariant measure. One of them is just that of {\em Riemannian foliations}, as a particular case of {\em equicontinuous laminations} (see \cite{Matsumotoequi}). In this context, the main result of this paper is the following:

\begin{theorem} \label{thm:uniqueergodicity}
If $\mathcal{F}$ is a minimal Riemannian foliation of a closed manifold $M$ by hyperbolic surfaces, then the horocycle flow $h_s^+$ on the unit tangent bundle $\hat M = T^1\mathcal{F}$ is  strictly ergodic --that is, minimal and uniquely ergodic. 
\end{theorem}

In order to prove that it is strictly ergodic, we make use of a theorem due to Y. Coudène which is precisely tailored for 
two continuous flows $g_t $ and $h_s^+$ satisfying the condition \eqref{geodesichorocyclic} above (see \cite{Coudene}). 

Coudène's theorem requires the existence of a measure $\mu$ on $\hat M$ that is invariant under both $g_t $ and $h_s^+$. In our case, this measure is the volume measure in $\hat M$ which is locally obtained by integrating the Liouville mesures on the unit tangent bundles to the leaves of $\F$ with respect to the transverse invariant volume $\nu$.

There is a bijection $\nu\mapsto \mu$ between transverse invariant measures for $\F$ and measures on the unitary tangent bundle $\hat{M}$ which are invariant under the $PSL(2,\R)$-action, that takes ergodic measures to ergodic measures. This is a general fact for any action of a  unimodular Lie group such as $PSL(2,\R)$, a proof of which can be found  in \cite[Theorem 5.2]{Connell-Martinez}.

Another relevant assumption in Coudène's theorem is that the horocycle flow $h_s^+$ is transitive. In our setting, we do not have a straightforward proof of this fact. Nevertheless, this follows from Moore's Ergodicity Theorem (which can be found in \cite[Chapter 2]{Zimmer}):
applying it to the $PSL(2,\R)$-action  on $\hat M$ preserving the measure $\mu$ described above, we know that the geodesic flow $g_t$ and both horocycle flows $h^\pm_t$ are ergodic with respect to $\mu$.
Since the foliation $\F$ is minimal, $\mu$ has total support.
All three flows have an orbit which is dense in the support of $\mu$ and hence they are topologically transitive.

The most technical hypothesis in Coudène's theorem is a semi-local condition that involves the long-term control of the geodesic orbits, as explained in Definition~\ref{definition:absolute_continuity}. The main technical tool that we introduce in this paper is the notion of {\em normal tube}, an object which arises from the geometry of the Riemannian foliation, see Definition~\ref{definition:tube}. We use it to prove that any geodesic in a leaf is followed closely, during all its past, by one and only one geodesic in any nearby leaf. This is the essential content of our main technical result, which is Proposition~\ref{prop:Riemannian}.
\medskip 

Section~\ref{section:Preliminaries_and_notation} introduces preliminary material, noticeably the precise statement of Coudène's theorem. Section~\ref{section:Homogeneous_G-Lie_foliations} contains both an example and a non-example. The example consists of Theorem~\ref{thm:uniqueergodicity} in the homogeneous case, which introduces the idea of the proof while avoiding its main difficulty. The non-example was introduced to shed further light on the semi-local condition used by Coudène, showing that it is not exclusive to the Riemannian setting. It also highlights the importance of the topological transitivity of the horocycle flow. Finally, Section~\ref{section:Riemannian} discusses Riemannian foliations and the proof of Theorem~\ref{thm:uniqueergodicity}.
\medskip

The authors would like to thank the referee for a suggestion that significantly improved the paper as well as for a very careful reading of the manuscript. We are also grateful to {\em Institut Henri Poincaré} for its hospitality and financial support during the initial preparation of this work as part of the {\em Research in Paris} program. The third author is grateful to the {\em Universidade de Santiago de Compostela} for its hospitality. This work was partially supported by Spanish MINECO/AEI Excellence Grant MTM2016-77642-C2-2-P and European Regional Development Fund, CSIC (UdelaR, Uruguay) and project PAPIIT IN106817 (UNAM, Mexico).

\setcounter{equation}{0}

\section{Preliminaries and notation}
\label{section:Preliminaries_and_notation}

\subsection{Geodesic and horocycle flows on hyperbolic surfaces}

Let $S$ be a hyperbolic surface, obtained as a quotient $S=\Gamma\backslash \H$ of the hyperbolic plane $\H$ by the action of a torsion-free discrete subgroup $\Gamma$ of the group of orientation-preserving isometries $Isom_+(\H)$. The group $Isom_+(\H)$, which is isomorphic to $PSL(2,\R)$, acts freely and transitively on the unit tangent bundle $T^1\H$ of the hyperbolic plane. Therefore, we will identify $Isom_+(\H)\simeq PSL(2,\R)\simeq T^1\H$. Under these identifications, the unit tangent bundle to $S$ is $T^1S\simeq \Gamma\backslash PSL(2,\R)$. This is the phase space of the geodesic and horocycle flows on $S$, which we will denote by $g_t$ and $h^\pm_t$. To simplify the notation $h^+_t$ will simply be called $h_t$.

When  $T^1S$ is seen as $\Gamma\backslash PSL(2,\R)$, the geodesic flow  $g_t$ is the right action of the diagonal group 
$$D=\left\{ \left(
\begin{array}{cc}
 e^{t/2}&0\\
 0&e^{-t/2}
\end{array}
\right):\ t\in\R
\right\},$$
and the horocycle flows  $h_s^\pm$ are the right actions of the unipotent subgroups 
$$U^+=\left\{ \left(
\begin{array}{cc}
 1&s\\
 0&1
\end{array}
\right):\ s \in\R
\right\} \mbox{ and }
U^-=\left\{ \left(
\begin{array}{cc}
 1&0\\
 s&1
\end{array}
\right):\ s \in\R
\right\},$$
 satisfying $g_t \scirc h_s^+ = h_{se^{- t}}^+ \scirc g_t$ and $g_t \scirc h_s^- = h_{se^t}^- \scirc g_t$. The joint action of $g_t$ and $h^-_s$ is of course the action of the Borel group 
$$B=\left\{ \left(
\begin{array}{cc}
 a&0\\
 b&a^{-1}
\end{array}
\right):\ a,b\in\R, \ a > 0
\right\}.$$

\begin{remark}
 \label{remark:horocycles_are_stable-manifolds} \begin{list}{\labelitemi}{\leftmargin=0pt} 
 \item[~(i)]  The (un)stable horocycles are the strongly (un)stable manifolds for the geodesic flow in $T^1S$, that is,
 $$h^\pm_\R(u)=\{v\in T^1S:\ d(g_t(u), g_t(v))\to 0 \mbox{ as } t\to \pm\infty\},$$
 where $d$ is the distance in $T^1S$.  
 
 \item[~(ii)] If $u$ is a point in $T^1S$, its {\em local weakly unstable manifold} of size $\varepsilon$ is the set
$$W^{wu}_\varepsilon (u)=\{v\in T^1S:\ d(g_{-t}(u), g_{-t}(v))<\varepsilon \; , \; \forall t\geq 0\}.$$
\end{list} 
\end{remark}

We will assume that the injectivity radius of $S$ is bounded below by a positive constant. More precisely, that there exists an $\varepsilon_0>0$ such that the restriction of the 
covering map $T^1\H\to T^1S$ to any ball of radius $\varepsilon_0$ is injective --and therefore a local isometry of class $C^\infty$.

\begin{remark}\label{remark:weakly_unstable_manifold}
  For any positive $\varepsilon \leq \varepsilon_0$, there is an open neighbourhood $B_\varepsilon$ of the identity in the group $B$ such that 
 $$W^{wu}_\varepsilon(u)=u\cdot B_\varepsilon.$$
\end{remark}

\subsection{Foliations by hyperbolic surfaces}

A smooth closed manifold $M$ admits a \emph{smooth foliation } $\F$ of dimension $n$ and codimension $m$ if $M$ is covered by open sets $U_i$ which are homeomorphic to the product of the unit open disks $D^n$ in $\R^n$ and $D^m$ in $\R^m$, in a way that verifies the conditions that follow. If we denote by $\varphi_i : U_i \to D^n \times  D^m$, the corresponding foliated chart, each \emph{distinguished open set} $U_i$ splits into \emph{plaques}  $P_i = \varphi_i^{-1}(D^n  \times  \{y\})$. Each point $y \in D^m$ can be identified with the point $\varphi_i^{-1}(0,y)$ in the \emph{local transversal} $T_i = \varphi_i^{-1}(\{0\} \times  D^m)$. The plaques glue together to form maximal connected surfaces called \emph{leaves}. The disjoint union $T = \sqcup \, T_i$ is a \emph{complete transversal} that meets all the leaves. In addition, the changes of charts
$\varphi_{i} \scirc \varphi_j^{-1} : \varphi_j(U_i \cap U_j) \to \varphi_i(U_i \cap U_j)$
are smooth diffeomorphisms given by
\begin{equation}
\label{ec:cambio_coor}
\varphi_{i} \scirc \varphi_j^{-1}(x,y)
= (\varphi_{ij}(x,y),\gamma_{ij}(y)).
\end{equation}
Then $\gamma_{ij}$ is an smooth diffeomorphism between open subsets of $T_j$ and  $T_i$. 
As is usual, we shall assume that any foliated atlas $\{ (U_i,\varphi_i) \} _{i\in I}$ satisfies the following conditions:

\begin{list}{\labelitemi}{\leftmargin=20pt}

\item[(i)\;\;]  the cover $\mathcal{U} = \{U_i\}_{i\in I}$ is locally finite, hence finite if $M$ is compact,

\item[(ii)\;]  each open set $U_i$ is a relatively compact subset of some foliated chart,

\item[(iii)] if $U_i \cap U_j\neq \emptyset$, there is a foliated chart containing $\overline{U_i \cap U_j}$ and then each plaque of $U_i$ intersects at most one plaque of $U_j$.
\end{list}
Sometimes it is more convenient to define $\F$ by a \emph{foliated cocycle} $(\{(U_i,\pi_i)\}, \{g_{ij}\})$ where $\{U_i\}$ is an open covering, 
$\pi_i : U_i \to T_i $ are smooth surjective submersions, and $g_{ij}$ are locally constant maps sending each point $p \in U_i \cap U_j$
to a local smooth diffeomorphism $g_{ij}(p)$ from $T_j$ to $T_i$ such that
\begin{equation}
\label{ec:cociclo}
\pi_i(p) = g_{ij}(p)(\pi_j(p)).
\end{equation}
The pseudogroup of local diffeomorphisms of $T$ generated by $\gamma_{ij}$ (or equivalently, by  $g_{ij}(p)$) is called the \emph{holonomy pseudogroup $\mathcal{P}$.}
\medskip 

Let $M$ be a compact manifold endowed with a foliation $\F$. If $g$ is a Riemannian metric on $M$, its restriction to the leaves endows each leaf  $L$ with a Riemannian metric $g_L$. \ If $T\F$ is the subbundle of $TM$ consisting of vectors tangent to $\F$, then we have a Riemannian metric tensor $g_T$ defined on $T\F$. 
Moreover, the normal bundle $N\F = TM/T\F$ can be naturally identified with the subbundle $T^\perp\F$ consisting of vectors orthogonal to $\F$. We denote by $g_N$ the corresponding Riemannian metric tensor on $N\F$.
The Riemannian metric $g$ splits then  into the sum
\begin{equation} \label{ec:metricdecomp}
g = g_T \oplus g_N
\end{equation}
on $TM = T\F \oplus N\F$.
\medskip 

Assume that $\F$ is a \emph{foliation by surfaces}, i.e. the leaves of $\F$ have dimension $n=2$. Poincar\'e's uniformisation theorem tells us that in the conformal class of Riemannian metric $g_L$ induced by $g$ on every leaf $L$, there will always be a metric of constant curvature $1$, $0$ or $-1$.
We say that $\F$ is a \emph{foliation by hyperbolic surfaces} if the metric given by the uniformisation has curvature -1 on all leaves. This turns out to be a purely topological condition, independent of the choice of Riemannian metric $g$.
A theorem by A. Candel and A. Verjovsky (see \cite{Candel} and \cite{Verjovsky}) says that every leafwise metric of constant curvature $-1$ varies, in the direction transverse to the foliation, continuously in the smooth topology.
\medskip 

The Riemannian metric $g_T$ along the leaves of
$\F$ allows us to consider unit vectors in $T\F$. 
We therefore have:

\begin{definition}
\label{definition:unit_tangent_bundle}
The \emph{unit tangent bundle} $T^1\F$ 
is the circle bundle $\mathfrak{p} : T^1\F \to M$ whose fibre at $x$ is the set of unit vectors tangent to the leaf of $\F$ through $x$. The manifold $\hat M=T^1\F$ has a foliation $\hat\F$ 
of dimension $3$ and codimension $m$,  
whose leaves are the unit tangent bundles of the leaves of $\F$.
\end{definition}

The Riemannian metric $g_L$ on each leaf $L$ induces the corresponding Sasaki metric on the leaf $T^1L$ of $\hat\F$. Thus, the Riemannian metric $g_T$ along the leaves of $\F$ induces in a natural way a Riemannian metric $\hat g_T$ along the leaves of $\hat \F$. In fact, there is a lifted Riemannian metric $\hat g$ on $\hat M$ such that the projection $\mathfrak{p} : \hat M \to M$ becomes a Riemannian submersion. Moreover, as in \eqref{ec:metricdecomp}, $\hat g$ splits into the sum $\hat g = \hat g_T \oplus \hat g_N$, where $\hat g_N$ is the Riemannian metric tensor which makes $\mathfrak{p}_\ast : N\hat\F \to N\F$ an isometric bundle morphism. If $g_T$ is the uniformised Riemannian metric along the leaves of $\F$, then $\hat g_T$ varies only continuously in the transverse direction to $\hat\F$.
The geodesic and horocycle flows on the individual leaves of $\hat\F$, when considered together, constitute the {\em foliated geodesic and horocycle flows}, which are continuous flows on $\hat M$. As in the case of surfaces, we will call them $g_t$ and $h^\pm_s$.   
When identifying $T^1\H\simeq PSL(2,\R)$,
the foliation  $\hat\F$ becomes the foliation by orbits of a continuous locally-free right $PSL(2,\R)$-action on $\hat M$. The foliated geodesic and horocycle flows 
are given by the action of the one parameter subgroups $D$, $U^+$ and $U^-$.

\subsection{Coudène's theorem} \label{subsection:Coudene}

The key tool in our proof of unique ergodicity for the foliated horocycle flow in Riemannian foliations is a theorem due to Y. Coudène, which is the object of his paper \cite{Coudene}. We will state it first and introduce the necessary definitions next.

\begin{coudene}
Let $X$ be a compact metric space, $g_t$ and $h_s$ two continuous flows on $X$ which
satisfy the relation: $g_t \scirc h_s = h_{se^{-t}} \scirc g_t$. Let $\mu$ be a Borel probability measure invariant
under both flows, which is absolutely continuous with respect to $W^{wu}$, and with full
support. Finally assume that the flow $h_s$ admits a dense orbit. Then $h_s$ is uniquely
ergodic.
\end{coudene}

The most technical assumption is 
the absolute continuity of $\mu$ with respect to $W^{wu}$, which we will now explain. 

\begin{definition}\label{definition:weakly_unstable_manifold_for_coudene's_theorem}
 Given $\varepsilon>0$ and a point $p\in X$, the \emph{local weakly unstable manifold} at $p$ of size $\varepsilon$ for the flow $g_t$ is the set
 $$W^{wu}_\varepsilon(p) = \{q\in X:\ d(g_{-t}(p),g_{-t}(q))<\varepsilon \; , \; \forall t\geq 0\}.$$
\end{definition}

\noindent It is, of course, in general not a manifold.

\begin{definition}[\cite{Coudene}] \label{definition:absolute_continuity}
 A probability measure $\mu$ invariant under the flow $h_s$ is \emph{absolutely continuous with respect to $W^{wu}_\varepsilon$} if the following two conditions hold:
 \begin{enumerate}
 \item {\bf Local structure of} $W^{wu}_\varepsilon$: For every point $p$ and every $\varepsilon>0$ there exists $\delta>0$ such that for all 
$q \in W^{wu}_\varepsilon(p)\cap B(p,\delta)$ and 
for all  $s \in (-1,1)$, the intersection $W^{wu}_\varepsilon(h_s(p))\cap h_{(-2,3)}(q)$ consists of a single point $q'$. 
Moreover, the map  

  \begin{center}
 \begin{tikzcd}[column sep =5ex]
 H : (q,s) \in (W^{wu}_\varepsilon(p)\cap B(p,\delta))\times (-1,1)  \arrow[mapsto]{r} &  q' \in X \end{tikzcd}
 \end{center}
 is a homeomorphism onto a neighbourhood of $p$.
 
  \item {\bf Local structure of the measure} $\mu$: In these coordinates, the measure $\mu$ disintegrates as $d\nu_s(q)\otimes ds$, where $d\nu_s$ is the conditional probability measure on $(W^{wu}_\varepsilon(p)\cap B(p,\delta))\times \{s\}$.

 \end{enumerate}
\end{definition}

 The first condition says that the $W^{wu}_\varepsilon$ are transverse to the flow $h_s$ and the local foliation they define has no holonomy. The second says that, when disintegrated with respect to the partition given by the $W^{wu}_\varepsilon$, 
the (pseudo)image of $\mu$ is the arc length.
 
\setcounter{equation}{0}

\section{Two examples}
\label{section:Homogeneous_G-Lie_foliations}

In this section, we will show how to apply Coudène's theorem to prove the unique ergodicity of homogeneous $G$-Lie foliations described in the introduction, which we will also compare to a totally different example. We will start by explaining their construction. 

\subsection{An example}

 Let $H$ and $G$ be two Lie groups and $\hat D : H \to G$ a surjective homomorphism. 
 Assume that $H$ contains a cocompact discrete subgroup $\Gamma$. The foliation by the fibres of $\hat D$, which are right cosets of 
 the kernel $K$ of $\hat D$,  is invariant by the left action of $\Gamma$, and then
 induces a 
 foliation on the compact manifold $\Gamma \bs H$. In fact, this construction can be modified by taking the quotient of the right action of some compact subgroup $K_0$ of $K$.
Indeed, the compact manifold $M = \Gamma \bs H / K_0$ admits a 
foliation induced by the foliation on $H/K_0$ whose leaves are the fibres of 
$D : H/K_0 \to G$. The leaves are diffeomorphic to $\Gamma \cap K \bs K / K_0$. 
 According to \cite[Appendix E]{Molinobook}, any foliation obtained by this method is called a \emph{homogeneous ($G$-Lie) foliation}.

 If $K = PSL(2,\R)$, a theorem by \'E. Cartan \cite[Corollary C]{Samelson} implies that $H = PSL(2,\R) \times G$ up to isomorphism of Lie groups. Given any cocompact discrete subgroup $\Gamma$ of $H$, the natural projection $\hat D : PSL(2,\R) \times G \to G$  defines a homogeneous 
foliation on $\hat M = \Gamma \bs (PSL(2,\R) \times G)$. Taking  $K_0 = PSO(2,\R)$, we obtain a homogeneous 
foliation on the manifold
$$M = \Gamma \bs (PSL(2,\R)/PSO(2,\R) \times G) \simeq \Gamma \bs (\H \times G).$$

 A particular example is obtained by suspension when 
$\Gamma$ is the fundamental group of a compact hyperbolic surface $S = \Gamma \bs \H$ and $\rho :\Gamma \to G$ is a representation in a compact Lie group $G$.  Indeed, there is a diagonal action of $\Gamma$ on $\H\times G$ given by 
$$\gamma\cdot (x,g)=(\gamma\cdot x,\rho(\gamma) g),$$
where $\Gamma$ acts on $\H$ by deck transformations, and the horizontal foliation on \mbox{$\H \times G$} projects onto a homogeneous 
foliation on $M=\Gamma\bs (\H\times G)$.
 The leaves are diffeomorphic to $Ker \, \rho \bs \H$.

In  this homogeneous situation, all manifolds involved in the construction 
are equipped with Riemannian metrics defined in a natural way.
The foliation $\hat\F$ on the unit tangent bundle $\hat M = \Gamma \bs (PSL(2,\R) \times G) \simeq \Gamma \bs (T^1\H \times G)$ is induced by the horizontal foliation on $H = PSL(2,\R) \times G \simeq T^1\H \times G$. 
The hyperbolic metric on $\H$ induces a Sasaki metric on $T^1\H$. Under the identification of $T^1\H$ with $K=PSL(2,\R)$, the Sasaki metric becomes a metric $g_K$ on $K$ which is left-invariant. Endowing $G$ with a left-invariant metric $g_G$, the bundle $\hat M$
 is endowed with the Riemannian metric $\hat g$ induced by the sum $g_K \oplus g_G$. Clearly $\hat g$ projects onto a metric on $M$ which comes from the sum $g_\H \oplus g_G$ on $\H \times G$. A summary of all the relevant maps is given by the following diagram:

 \begin{equation} \label{ec:maps}
 \begin{tikzcd}[row sep=6ex, column sep =2ex]
  & T^1\H  \times G \arrow[swap]{dl}{\hat \pi} \arrow{dr}{\mathfrak{p}} \arrow[bend left=15]{drrr}{\hat D} &  & & \\ 
 \hat M \arrow[swap]{dr}{\mathfrak{p}} & & \H \times G  \arrow{rr}{D} \arrow{dl}{\pi} & \quad & G \\
  & M & & & 
 \end{tikzcd}
 \end{equation}
where all maps are Riemannian submersions (with the metrics that we have just described).
Since $\hat \pi : T^1\H\times G\to \hat M$ is actually a local isometry, 
 any contractible neighbourhood of a point in $\hat M$ is identical from a topological and metrical point of view to a neighbourhood of any of its preimages in $T^1\H\times G$. 
\medskip 

 Let $\varepsilon_0 > 0$ be small enough 
such that $\hat \pi$ is injective on any ball of radius $\varepsilon_0$ (and hence smaller than the injectivity radius of the surface $\Gamma\backslash \H$), and let $\varepsilon \leq \varepsilon_0/2$.
Denote by $B(e,\varepsilon)$ the ball in $G$ centred at $e$ of radius $\varepsilon$, and recall $B_\varepsilon$ is the open ball in the Borel group $B$ given by Remark~\ref{remark:weakly_unstable_manifold}.

With these notations, for the product metric in $T^1\H\times G$, we have:
 $$W^{wu}_\varepsilon (u,e) \subset \big( \{u\}\times B(e,\varepsilon)\big)\cdot B_\varepsilon =  W^{wu}_\varepsilon(u) \times B(e,\varepsilon) \subset W^{wu}_{2\varepsilon}(u,e)$$
for all $u \in T^1\H$. Since $\hat \pi$ is injective at this small scale, we have: 
 $$
W^{wu}_{\varepsilon}(\hat \pi(u,e)) =\hat \pi(W^{wu}_{\varepsilon}(u,e)) \subset\hat \pi \big(  W^{wu}_\varepsilon(u) \times B(e,\varepsilon) \big) 
\subset W^{wu}_{2\varepsilon}(\hat \pi(u,e)).
 $$
 Notice that all the above discussion refers to the {\em metric} and not the {\em algebraic} structure of $G$. Therefore, since $G$ acts transitively on itself by left translations which are isometries, the identity element $e$ does not play any special role:

\begin{proposition}  \label{prop:homogeneus}
 The local weakly unstable manifolds for the foliated geodesic flow  $g_t$ at any point $\hat u$ of $\hat M$ are given by
 $$
W^{wu}_{\varepsilon}(\hat u) \subset\hat \pi \big( W^{wu}_\varepsilon(u) \times B(k,\varepsilon) \big) 
\subset W^{wu}_{2\varepsilon}(\hat u)
 $$
 when $\hat u =\hat \pi(u,k)$ for some $(u,k) \in T^1\H \times G$. \qed
\end{proposition}

\noindent

It allows us to deduce the following result from Coud\`ene's theorem: 

\begin{proposition}[Theorem~\ref{thm:uniqueergodicity} for homogeneous $G$-Lie foliations] \label{thm:homogeneous}
If $\F$ is a  minimal homogeneous $G$-Lie foliation on a closed manifold $M$, then the horocycle flow $h_s$ on the unit tangent bundle $\hat M = T^1\F$ is  strictly ergodic. 
\end{proposition}

\begin{proof}
 As explained in the introduction, the horocycle flow $h_s$ is transitive, although we already know that $h_s$ is minimal in this case from \cite{Alcalde-Dal'Bo}. Therefore, all we have to prove that the volume measure $\mu$ that comes from the metric $g$ is absolutely continuous with respect to $W^{wu}$.

This condition is 
semi-local, but the orbits of $h_s$ being contractible, we can always reduce to contractible neighbourhoods which trivialise $\hat \pi$. So rather than proving it in $\hat M$ we will prove it in $PSL(2,\R)\times G$.  The Borel probability measure $\mu$ on $\hat M$ is induced (up to normalisation) by the  product $\mu_{PSL(2,\R)}\otimes \mu_G$ of the left Haar measures of $PSL(2,\R)$ and $G$.

To see that the first condition in Definition~\ref{definition:absolute_continuity} is satisfied, we will fix a point  $p \in T^1\H\times G$ and $\varepsilon > 0$. As has been remarked before, without loss of generality, we can take 
a point of the form $p=(u,e)$, with $u\in T^1\H=PSL(2,\R)$ and $e$ being the identity element in $G$. We will also consider a point $q \in W^{wu}_\varepsilon (p)\cap B(p,\delta)$, a time $s\in (-1,1)$, and we will see that the intersection 
$$W^{wu}_\varepsilon (h_s(p))\cap h_{(-2,3)}(q)$$
has exactly one point.

The point $q=(ub_0,k_0)$ for some $b_0\in B_\varepsilon$ and $k_0 \in B(e,\varepsilon)$. The horocycle flow is given by the action on the $PSL(2,\R)$ factor defined by right translation of the one parameter group $U^+=\{h_s\}_{s\in\R}$. Slightly abusing notation, we will write $h_s(p)=(uh_s,e)$, and similarly for horocycle orbits through other points. The intersection above has exactly one point if and only if there exist unique $l\in (-2,3)$, $b_1\in B_\varepsilon$ and $k_1\in B(e,\varepsilon)$ such that 
$$(uh_sb_1,k_1)=(ub_0h_l,k_0).$$
Clearly the only possibility for $k_1$ is $k_1=k_0$.
The equality of the first component gives 
$$
h_{-s}b_0 = 
\left( 
\begin{array}{rr}
1 & -s \\
0 & 1
\end{array} \right) 
\left( 
\begin{array}{cc}
\lambda_0 & 0 \\
t_0 & \lambda_0^{-1} 
\end{array} \right) 
 = 
 \left( 
 \begin{array}{cc} 
 \lambda_1 & 0 \\
 t_1 & \lambda_1^{-1} 
 \end{array} \right) 
\left( 
\begin{array}{rr}
1 & -l \\ 
0 & 1 
\end{array} \right)  = b_1h_{-l}
$$
which is an equation in $PSL(2,\R)$. The left hand side is known. Now, there is a unique choice of $b_1\in B$ and $h_{-l}\in U^+$ such that this equation is satisfied,  namely 
$$
\lambda_1 =  \lambda_0 - st_0 \, , \quad t_1 = t_0 \, , \quad \mbox{and} \quad  l =  \frac{s}{\lambda_0\lambda_1}.
$$
The time $s$ is arbitrary in $(-1,1)$ so it cannot be chosen. But taking a small enough $\delta$, $b_0$ is very close to the identity  so 
$t_0 \approx 0$ and $\lambda_0 \approx 1$. 
This is enough to guarantee that $b_1\in B_\varepsilon$ and $l\in (-2,3)$.

This defines the map 
$$
\begin{array}{rcl}
(W_\varepsilon^{wu}(p)\cap B(p,\delta))\times (-1,1) &\to & PSL(2,\R)\times G\\
(q,s) &\mapsto & q_s,
\end{array}
$$
where $q_s$ is the only intersection point of $h_{(-2,3)}(q)$ and $W^{wu}_\varepsilon(h_s(p))$. 
It is clearly injective and continuous. A close look at Proposition~\ref{prop:homogeneus} shows that it is also surjective onto an open neighbourhood of $p$, and since all spaces are locally compact it is a homeomorphism onto its image.

We now have to verify that the second condition in Definition~\ref{definition:absolute_continuity} holds. That is, that the disintegration of the measure  
$\mu_{PSL(2,\R)}\otimes \mu_G$ with respect to the partition given by the above coordinates is of the form
$$
 \mu_{PSL(2,\R)}\otimes \mu_G =\nu_s(q)\otimes ds,$$
where $\nu_s$ is the conditional (probability) measure along 
$(W^{wu}_\varepsilon(p)\cap B(p,\delta))\times \{s\}$ and $ds$ is the Lebesgue measure in $(-1,1)$. That is, we have to see that 
$$\mu_{PSL(2,\R)}\otimes \mu_G \big((W^{wu}_\varepsilon(p)\cap B(p,\delta))\times (s',s'')\big) =s''-s',$$ for all $-1<s'<s''<1$.

The $s$ parameter does not affect the measure on the $G$-factor, so it is enough to verify the condition on $PSL(2,R)$ --that is, as if we were working in the hyperbolic plane.
We are simply considering the disintegration of the Haar measure $\mu_{PSL(2,\R)}$ in (an open subset of) $PSL(2,\R)$ with respect to the foliation given by weakly unstable manifolds of points in a fixed horocycle orbit. 
 Thus, we are disintegrating with respect to 
$$
W^{wu} (h_{s}(u)) = u
\left(\begin{array}{cc}                                                                                                                                                                                                                                                                                                                                                                                                                                                                                                                                                                                                1&s \\                                                                                                                                                                                                                                                                                                                                                                                                                                                                                                                                                                                            0&1
\end{array}
\right) \cdot B 
$$
for $s \in (s',s'')$.
This set has $\mu_{PSL(2,\R)}$-measure which only depends on $s''-s'$ simply because the left Haar measure $\mu_{PSL(2,\R)}$ is also invariant under right translations in $PSL(2,\R)$. 
\end{proof}

\begin{remark}
 \label{remark:proof_of_main_theorem_works_in_general_case}
 Notice that the above proof does not use the algebraic structure of $G$. Namely, once we have Proposition~\ref{prop:homogeneus}, the fact that the transverse structure of $\F$ is modelled on the Lie group $G$ is no longer necessary.
\end{remark}

\subsection{A non-example}

Let $Sol^3$ be the 3-dimensional solvable Lie group (diffeomorphic to $\R^3$) with the multiplication law given by 
$$
(x,y,z).(x',y',z') = (x+e^zx', y+e^{-z}y',z+z').
$$
If $A \in SL(2,\Z)$ is a hyperbolic matrix (with $|tr(A)| > 2$), then $A$ is conjugate to a diagonal matrix 
$$
\left(
\begin{array}{ll}
\lambda & 0 \\
0 & \lambda^{-1}
\end{array} \right),
$$
where $\lambda > 1 > \lambda^{-1}$ up to multiplication by $-Id$. Using the eigenvector basis and multiplying the third coordinate by $1/log\lambda$,  the group $\Z^3$ with the multiplication law given by
$$
(m,n,p).(m',n',p') = ((m,n) + A^p(m',n'), p+p')
$$
can be realised as a cocompact discrete subgroup $\Gamma$ of $Sol^3$. The quotient manifold $M=\Gamma \bs Sol^3$ is a torus fibre bundle over $S^1$ with linear monodromy $A$, which is usually called $T^3_A$, see for example \cite{Alcalde-Dal'Bo} and \cite{GhysSergiescu}. It is endowed with a codimension one foliation $\F$ defined by the 
1-form on $M$ which is induced by the left invariant 1-form $\omega = e^z dy$ on $Sol^3$. As is usual, we will use the somewhat abusive notation identifying left invariant 1-forms and vector fields on $Sol^3$ with their projections on $M$.  This foliation is induced by the left invariant vector fields 
$$
X = e^z \vf{x} \quad \mbox{and} \quad Z = - \vf{z},
$$
generated by the flows 
$$
h^+_s(x,y,z) = (x,y,z).(s,0,0) = (x+e^zs,y,z)
$$
and
$$
\quad g_t(x,y,z) = (x,y,z).(0,0,t) = (x,y,z+t)
$$
respectively. To justify this notation, let us observe that its Lie bracket $[X,Z] =Z$ and hence $\F$ is defined by a locally free $B^+$-action. This is actually the right $B^+$-action obtained from the realisation of the affine group $B^+$ (identified to the semidirect product $\R \rtimes \R^\ast_+$) as a Lie subgroup of $Sol^3$ which sends $(a,b)$ to $(a,0,log \, b)$. So the one-parameter subgroups of $B^+$ generated by $X$ and $Z$ are given by 
$$
h_s^+(a+bi) = a+bs+bi \quad \mbox{and} \quad g_t(a+bi) = a + e^tbi,
$$
where $(a,b) \in \R \rtimes \R^\ast_+$ is identified with 
$$
a+bi = \frac{\sqrt{b} \, i + a/ \sqrt{b}}{1/\sqrt{b}} \in \H.
$$

On the other hand, 
$$
g = e^{-2z}dx^2 + e^{2z}dy^2 + dz^2
$$
is a left invariant Riemannian metric on $Sol^3$, and its restriction
$$g_L = e^{-2z}dx^2 + dz^2 = \frac{da^2 + db^2}{b^2}$$
to each leaf $L$ is precisely the Poincar\'e metric $g_\H$ (up to the coordinate change $b = e^z$). 
The third left invariant vector field $Y = e^{-z}\vf{y}$ is orthogonal to the foliation $\F$ and generated by the orthogonal flow
$$
v_s(x,y,z) = (x,y,z).(0,s,0) = (x,y+e^{-z}s,z). 
$$
As before, since the Lie bracket $[Y,Z] = -Y$, we deduce that $g_{-t} \scirc v_s = v_{se^{-t}} \scirc g_{-t}$ so that the flow $v_s$ play the role of the horocycle flow $h_s^-$. 
The biinvariant Haar measure 
$$
dx \wedge dy \wedge dz =  b (\frac{da \wedge db}{b^2}) \wedge (- dy) =  e^z vol_\H \wedge (- dy)
$$
induces a probability measure $\nu$ on $M$. Since $b= e^z$ is a harmonic function on $\H$, this measure $\nu$ is harmonic. In fact, according to a theorem by B. Deroin and V. Kleptsyn \cite{Bertrand}, this is the unique harmonic measure of $\F$. A theorem by 
Y. Bakhtin and the third author \cite{Bakhtin} proves that $\nu$ lifts to a unique measure $\mu$ on the unit tangent bundle $\hat M = T^1 \F$ which is 
invariant under both foliated geodesic $g_t$ and horocycle $h_s^+$ flows. 
\medskip 

As proved in \cite{Alcalde-Dal'Bo}, the horocycle flow $h_s^+$ is not minimal because orbit closures reduce to integral surfaces (diffeomorphic to $T^2$) of the vector fields $X$ and $Y$. In fact, contrary to the previous example, this flow is neither uniquely ergodic, nor ergodic with respect to to a volume. However, conditions 1 and 2 in Definition~\ref{definition:absolute_continuity} hold in this case. 

Indeed, the manifold $\hat M$ is the quotient of the unitary tangent bundle to the lifted foliation, and this bundle is made up of pairs $(p,u)$ where $p= (x,y,z)$ is an element of $Sol^3$ and $u$ is a unitary vector in the plane generated by $X_p$ and $Z_p$. The measure $\mu$ is induced by the product of the Haar measure $\mu_{PSL(2,\R)}$ modified with the density $e^z$ and the transverse volume $-dy$. With the obvious abusive notation, already used above, the local weakly unstable manifolds
$W^{wu}_\varepsilon(p,u)$ are made up of unitary tangent vectors $(p',u')$ such that $d(g_{-t}(p,u),g_{-t}(p',u'))<\varepsilon$ for all $t\geq 0$. If 
$p'=(x',y',z')$, then 
$p'' = v_{e^{z'}(y-y')}(x',z',y') = (x',y,z')$ belongs to the leaf passing through $p = (x,y,z)$. Since 
\begin{eqnarray*}
v_{e^{z'-t}(y-y')} ( g_{-t} (x',y',z') ) & = & v_{e^{z'-t}(y-y')}(x',y',z'-t) \\ 
&  = & (x', y'+e^{-(z'-t)}e^{z'-t}(y-y'),z'-t) \\ 
 & = & (x', y',z'-t) \\ 
 & =  & g_{-t} ( v_{e^{z'}(y-y')}(x',z',y') ) = g_{-t}(p''),  
\end{eqnarray*}
we have
\begin{eqnarray*} 
d \big( g_{-t}(p,u),g_{-t}(p',u') \big)& \leq & 
d  \big( g_{-t}(p,u),g_{-t}(p'',u'') \big) \\ 
& + & d  \big( g_{-t}(v_{e^{z'-t}(y-y')}(g_{-t}(p',u'))),g_{-t}(p',u')\big),
\end{eqnarray*}
where $d \big( g_{-t}(p,u),g_{-t}(p'',u''))$ is just the distance in the leaf passing through $(p,u)$ and 
$$
d \big( g_{-t}(v_{e^{z'-t}(y-y')}(g_{-t}(p',u'))),g_{-t}(p',u')\big)  \longrightarrow 0
$$
as $t \to +\infty$. Taking $\varepsilon_0 > 0$ and $\varepsilon \leq \varepsilon_0/2$ as before, for $t$ big enough, Proposition~\ref{prop:homogeneus} remains valid in $\hat M$ and therefore the first condition in Definition~\ref{definition:absolute_continuity} holds in this case. With respect to the second condition, let us remark that the contraction of the volume form on $\hat M$ with the vector fields $Y$ and $Z$ is the 1-form $e^{-z}dx$ which corresponds to the length element $ds$. 

\begin{remark} \label{remarkCouden}  \begin{list}{\labelitemi}{\leftmargin=0pt} 
\item[~(i)] Why does Coud\`ene's theorem not apply in this case? The reason is simply  that the foliated horocycle flow $h_s^+$ is not transitive.

\item[~(ii)] If we consider the transversely homographic example described in \cite[Example 4.1]{Alcalde-Dal'Bo}, or more generally the suspension of any faithful discrete representation of a torsion-free Fuchsian group $\Gamma$ into $PSL(2,\R)$, the unique minimal invariant subset $\mathcal{M}$ of the right $B^+$-action is also the unique minimal invariant subset of the horocycle flow $h_s^+$. In this case, Coud\`ene's theorem can still be applied to obtain the strict ergodicity of $h_s^+$ in restriction to $\mathcal{M}$. It is not difficult to give a supplementary argument 
to prove that $h_s^+$ has no invariant measures outside of $\mathcal{M}$, and therefore that $h^+_s$ is uniquely ergodic. Anyway, Coud\`ene's theorem is useful beyond the case of foliations with transverse invariant measure. 

\item[~(iii)] This counterexample also reveals a certain analogy with the case of skew-products studied by Furstenberg in \cite{Furstenberg1} in the sense that ergodicity seems to imply unique ergodicity when the foliation is uniquely ergodic.
\end{list} 
\end{remark}

\setcounter{equation}{0}

\section{Riemannian foliations} 
\label{section:Riemannian}

Our proof of 
Proposition~\ref{thm:homogeneous} is based on the fact that 
the 
maps $\pi : \H \times G \to M$ and $\hat \pi : T^1\H\times G\to \hat M$ are local isometries when we endow 
these manifolds with the natural Riemannian metrics which we have in the homogeneous case. 
 In this section, to prove Theorem~\ref{thm:uniqueergodicity}, we will replace these global `normal unwrappings' by partial `normal unwrappings', called \emph{normal tubes}, respecting the transverse metric structure of Riemannian foliations. We will start by recalling the definition of this kind of foliations.

\subsection{Riemannian foliations}
\label{subsection:Riemannian_foliations}

\begin{definition} \label{definition:Riemannian_foliation}
A \emph{Riemannian foliation} $\F$ of codimension $m$ on a closed manifold $M$ is given by a foliated cocycle $(\{(U_i,\pi_i)\}, \{g_{ij}\})$ with values in the pseudogroup of local isometries of a Riemannian manifold  $N$ of dimension $m$. 
\end{definition}

 If $g_N$ is the Riemannian metric on $N$, then each distinguished open set $U_i$ admits a Riemannian metric $g_i$ such that the distinguished map $\pi_i : U_i \to T_i$ is a Riemannian submersion,  where $N=\cup_i T_i$. Then the bundle map $(\pi_i)_\ast : N\F|_{\textstyle U_i} \to TN|_{\textstyle T_i}$ induces an isometry from  the fibre $N_p\F$ to the tangent space $T_yN$ for each point $p \in U_i$ with $y = \pi_i(p) \in T_i \subset N$. Since the transition maps are isometries, this defines a smooth Riemannian metric tensor on $N\F$, denoted again by $g_N$. Writing $TM=T\F\oplus N\F$, we can complete $g_N$ to a smooth Riemannian metric $g$ on $M$. It is called a {\em bundle-like metric} for $\F$. Then $\F$ has the following properties (see \cite{Reinhart}):

\begin{list}{\labelitemi}{\leftmargin=20pt}

\item[(1)] Any two orthogonal curves with the same local projection have the same length and therefore the local distance between leaves becomes constant. 

\item[(2)] A geodesic is orthogonal to the foliation at one point if and only if it is orthogonal at every point. 

\end{list}

In fact, both properties characterise Riemannian foliations. The structure of Riemannian foliations has been described by P. Molino in \cite{Molino} (see also \cite{Molinobook}). 

 The Riemannian volume on $N$ is then invariant by the holonomy pseudogroup $\mathcal{P}$. Although we will not make use of this fact, it is actually the unique transverse invariant measure $\nu$ for $\F$ if $\F$ is minimal, see \cite{Matsumotoequi} for a detailed proof in the more general setting of equicontinuous laminations.

Now we will consider a Riemannian foliation $\F$ by hyperbolic surfaces on a closed manifold $M$. 
 Let $g$ be a bundle-like metric on $M$ which decomposes into the sum $g = g_T \oplus g_N$ on $TM = T\F \oplus N\F$ where $g_N$ is a smooth Riemannian metric tensor on $N\F$ glueing together the local metric tensors $g_i$. 
The metric $g$ determines a conformal class on each leaf $L$, and there is a unique complete metric of constant curvature $-1$ on this conformal class, which will be denoted by $g_L$. Glueing all these hyperbolic metrics we get the Riemannian metric tensor $g_T$ on $T\F$, which is not known to be smooth but only continuous in the smooth topology on $M$.

\subsection{Normal tubes}
\label{tubes_Riemannian_foliations}

The existence of normal tubes is true for minimal Riemannian foliations of any dimension, but we will continue to restrict to the case of surfaces for convenience:

\begin{definition} \label{definition:tube}
Let $\F$ be a Riemannian foliation on a closed manifold $M$ defined by a foliated cocycle $(\{(U_i,\pi_i)\}, \{g_{ij}\})$ with values in the pseudogroup of local isometries of a Riemannian manifold $N$. Assume the leaves of $\F$ are dense hyperbolic surfaces, and denote by $\D$ the unit disk in $\C$ and by $(T_0,y_0)$ a pointed open subset of $N$ with the metric induced by $g_N$. 
A \emph{normal tube} is a surjective foliated smooth map 
$$
\tau : (\D \times T_0, \mathcal{H}) \longrightarrow (M,\F),
$$
where $\mathcal{H}$ denotes the horizontal foliation of $\D \times T_0$, such that
\begin{list}{\labelitemi}{\leftmargin=5pt}

\item[(i)\;\;] $\tau$ is a local diffeomorphism, 

\item[(ii)\;]  $\tau |_{\D \times \{y\}}$ is the universal covering map of the leaf $L_y$ passing through $y$,   

\item[(iii)] $\tau |_{\{x\} \times T_0}$ is an isometric embedding into a local integral submanifold of $N\F$ passing through $\tau(x,y_0)$.  By a {\em local integral submanifold} we mean the image of the exponential map. In other words, for any ray $\ell$ in $T_0$ starting at the base point $y_0\in T_0$, the image $\tau(\{x\}\times\ell)$ is an isometrically embedded geodesic with respect to the bundle-like metric. In particular, the image of the ray is tangent to $N\F$ (although the whole image $\tau (\{x\}\times T_0)$ may not be). An important property is that for any point $y\in T_0$, $\tau(\D\times\{y\})$ is contained in a leaf of $\F$.

\end{list}

\end{definition}

 Although the local integrability of $N\widetilde\F$ can be used to `partially unwrap' any minimal Riemannian foliation, it will be more convenient to use groupoids as `global unwrapping' of foliations:

\begin{definition}[\cite{Winkelnkemper}]
For any smooth foliation $\F$ on a closed manifold $M$, the \emph{homotopy groupoid $\Pi_1(\F)$} is a Lie groupoid obtained as the quotient of the space $\mathcal{P}(\F)$ of paths tangent to $\F$ (with the smooth topology) by the homotopy relation. 
\end{definition}

Next we describe the differentiable structure of $\Pi_1(\F)$. For this structure, the maps $\alpha: \gamma \in  \mathcal{P}(\F) \mapsto \gamma(0) \in M$ and 
 $\beta: \gamma \in  \mathcal{P}(\F) \mapsto \gamma(1) \in M$ induce a pair of surjective submersions $\alpha : \Pi_1(\F) \to M$ and 
$\beta : \Pi_1(\F) \to M$. 
When the foliation is Riemannian, $\alpha$ and $\beta$ are Riemannian submersions, and therefore locally trivial bundles.

To endow $\Pi_1(\F)$ with a smooth structure, we will use tubes of paths
from \cite{AH}: for every pair of distinguished open sets $U \simeq P \times T$ and $U' \simeq P' \times T'$, we call \emph{tube of paths joining $U$ and $U'$} a map 
$$
\Theta : U \times P' \times [0,1] \to M
$$
such that:

\begin{list}{\labelitemi}{\leftmargin=20pt}

\item[(i)\;\;] $\Theta^0(x,y,x') = \Theta(x,y,x',0) = (x,y)$;

\item[(ii)\;]$\Theta_{(x,y,x')} : t \in [0,1] \mapsto \Theta(x,y,x',t) \in M$ is a path tangent to $\F$;

\item[(iii)] $\Theta^1(x,y,x') = \Theta(x,y,x',1)$ belongs to the local transversal in $U'$ passing through $x'$. 

\end{list}
If we identify $T$ and $T'$ with local transversals in $U$ and $U'$, the map
$$
\Theta^1 : \{x_0\} \times T \times \{x'_0\} \to T' \subset U'
$$
is the holonomy transformation corresponding to the path $\Theta_{(x_0,y_0,x_0')}$ tangent to $\F$ which joins the point 
$p = (x_0,y_0) \in U$ with the point $p' = (x'_0,y'_0) \in U$, where $y'_0 = \Theta(x_0,y_0,x'_0,1)$. 
Each tube of paths $\Theta$ defines a continuous map $$\hat{\Theta} : U \times P' \to \mathcal{P}(\F)$$ which projects in a local chart for the homotopy groupoid $\Pi_1(\F)$. When seen in these charts, it is clear that $\alpha$ and $\beta$ are submersions.
By construction, the fibre $\alpha^{-1}(p)$ is the universal cover of the leaf  
passing through $p$ and 
the restriction $\beta|_{\alpha^{-1}(p)} : \alpha^{-1}(p) \to \beta(\alpha^{-1}(p))$ is the covering map. 

Using Molino's theory, it was proved in \cite{Alcalde89} that if $\F$ is a Riemannian foliation, then $\alpha : \Pi_1(\F) \to M$ is a locally trivial bundle, see also \cite{GGHR} for another approach. This will be used in the proof of the next Proposition.

\begin{proposition} \label{tube_Riemannian}
Any minimal Riemannian foliation $\F$ by hyperbolic surfaces on a closed manifold $M$ admits a normal tube.
\end{proposition}

\begin{proof}
 Let $\pi_0 : U_0 \to T_0$ be a distinguished map defined on a contractible open subset $U_0\simeq P_0\times T_0$ of $M$. Then the projection $\alpha : \Pi_1(\F) \to M$ is trivialised over $U_0$, that is, there is a smooth diffeomorphism 
 $\varphi : \D \times U_0 \to \alpha^{-1}(U_0)$ such that $\alpha(\varphi (x,p)) = p$. It follows that $\beta \scirc \varphi : \D \times U_0 \to M$ is a smooth submersion, which is also surjective by minimality of $\F$. Moreover, by construction, its restriction to $\D \times \{p\}$  is the universal covering of the leaf $L_p =  \beta(\alpha^{-1}(p))$ passing through $p$. 
 Since the map $\beta  \scirc \varphi$ sends each horizontal  leaf to a leaf of $\F$ and each local transversal in $U_0$ to a local transversal to $\F$, the map 
 $$
 \tau = \beta \scirc \varphi |_{\D  \times T_0} : \D \times T_0 \to M
 $$
still is a surjective smooth submersion. But $\tau$ is actually a local diffeomorphism by dimension reasons. Thus, conditions (i) and (ii) of Definition~\ref{definition:tube} are satisfied by $\tau$. If we restrict $\tau$ to $\{x\} \times T_0$, as it is contained in the unit space $M$, we retrieve the inclusion of $T_0$ into $M$, which is obviously an isometric embedding.
 \end{proof}

\begin{remark} \label{remark:tangent_metric_riemannian}  \begin{list}{\labelitemi}{\leftmargin=0pt} 
 \item[~(i)] From the previous description, it is clear that the foliations $\alpha^\ast\F$ and $\beta^\ast\F$ induced by $\alpha$ and $\beta$ on the homotopy groupoid $\Pi_1(\F)$ are equal and the resulting foliation $\boldsymbol{\F} = \alpha^\ast\F = \beta^\ast\F$ has the same transverse structure than $\F$. Given a bundle-like metric $g$ for $\F$, the Riemannian metric $\boldsymbol{g}$ that coincides with 
$\alpha^\ast g \oplus \beta^\ast g$ on $T\boldsymbol{\F}$ and $\alpha^\ast g = \beta^\ast g = \frac 1 2 (\alpha^\ast g \oplus \beta^\ast g)$ on $N\boldsymbol{\F}$ is also bundle-like for $\boldsymbol{\F}$. In particular, 
the map $\beta: \Pi_1(\F) \to M$ is a surjective Riemannian submersion, and the same happens with 
$\beta  \scirc \varphi$.

\item[~(ii)]  Consider two leaves $\D\times\{y\}$ and $\D\times\{y'\}$ of the horizontal foliation $\mathcal{H}$. The map $(x,y)\mapsto (x,y')$ is not in general an isometry with the metrics induced by $\tau^\ast g$. Assuming that $g$ has been already uniformised, each $\D\times\{y\}$ carries a complete metric $g_y$ of constant curvature $-1$. These $g_y$
vary continuously in the smooth topology, uniformly in $x\in \D$.
\end{list}
\end{remark} 

\subsection{Unique ergodicity of the horocycle flow $h^+_s$} \label{proofMainTheorem}

In this section, we will use these normal tubes 
 to prove 
Theorem~\ref{thm:uniqueergodicity}.
The key point of the proof is to extend Proposition~\ref{prop:homogeneus} to the Riemannian case. Let $\F$ be a minimal Riemannian foliation by hyperbolic surfaces on a closed manifold $M$, and let $\tau : \D \times T_0 \to M$ be a normal tube for $\F$ obtained from Proposition~\ref{tube_Riemannian}. As in Diagram~\ref{ec:maps}, we can consider the local diffeomorphism  
$$
\hat \tau : T^1\D \times T_0 \to \hat M
$$
given by 
$$
\hat \tau(u,y) = ((\tau|_{\D \times \{y\}})_{\ast x}v,y)
$$
for each $u \in T^1_x\D$ and each $y \in T_0$. We will fix a  bundle-like metric  $g$ for $\F$ which is hyperbolic on leaves.
With the notations introduced in Remark~\ref{remark:weakly_unstable_manifold} and Definition~\ref{definition:absolute_continuity}, we have: 

\begin{proposition}  \label{prop:Riemannian}
There is $\varepsilon > 0$ such that for any point $\hat u$ of $\hat M$ there is a homeomorphism $F:\D\times T_0\to \D\times T_0$ which preserves the leaves of $\mathcal{H}$ and is smooth along the leaves such that the local weakly unstable manifolds for the foliated geodesic flow  $g_t$  are given by
 $$
W^{wu}_{\varepsilon}(\hat u) \subset \hat \tau \big( \hat F \big( W^{wu}_\varepsilon(u) \times T_0\big) \big) 
\subset W^{wu}_{2\varepsilon}(\hat u)
 $$
 when $\hat u =\hat \tau(u,y_0)$,
$u$ any vector in $T^1\D$ and $y_0$ the prescribed base point of $T_0$,  with $\hat F$ being the homeomorphism induced by $F$ on $T^1\D\times T_0$.
 \end{proposition}

  
\begin{figure}
\centering
\begin{tikzpicture}[line cap=round,line join=round,>=triangle 45,x=1.0cm,y=1.0cm,scale=1.05]
\clip(-8.2,-1.
5) rectangle (3.7,4.7);
\draw [rotate around={0.:(-5.5,3.)},line width=0.4pt] (-5.5,3.) ellipse (1.8030830657844836cm and 1.0005541175362629cm);
\draw [rotate around={0.:(-5.5,0.)},line width=0.4pt] (-5.5,0.) ellipse (1.7871360391181077cm and 0.9715221162252411cm);
\draw [line width=0.4pt] (-7.301993922677331,3.034771593552408)-- (-7.286911510639987,0.015399653820113275);
\draw [line width=0.4pt] (-3.6972815960394922,3.02012198887777)-- (-3.71286396088189,0.);
\draw [rotate around={0.:(0.7958569922480621,2.925546542635657)},line width=0.4pt] (0.7958569922480621,2.925546542635657) ellipse (1.8030830657844823cm and 1.0005541175362622cm);
\draw [rotate around={0.:(0.7958569922480623,-0.07445345736434293)},line width=0.4pt] (0.7958569922480623,-0.07445345736434293) ellipse (1.7871360391181124cm and 0.9715221162252438cm);
\draw [line width=0.4pt] (-1.002187357889388,3.0002953070174363)-- (-0.9836469719311594,0.015236956886861108);
\draw [line width=0.4pt] (2.5954233563654627,2.988006777698276)-- (2.5777373256253635,0.);
\draw [->,line width=1.pt] (-5.44927,3.03) -- (-4.577194189836342,3.0341169419465897);
\draw [line width=0.4pt] (0.86,3.01)-- (0.8733333333333335,-0.003333333333334837);
\draw [line width=0.4pt,color=black]
(-7.301993922677331,3.034771593552408)-- (-3.6972815960394922,3.02012198887777);
\draw [line width=0.4pt] (-5.44927,3.03)-- (-5.26067,2.1979)-- (-5.26066780913006,1.5441625178294616)-- (-4.8974802122308345,1.1083374015503922)-- (-5.024595871145563,0.4727591069767496)-- (-4.833333333333334,-0.24482956899224875);
\draw [shift={(-5.499670200352144,5.560192815486925)},dotted,line width=0.4pt,color=black]
plot[domain=4.401497905183707:5.023161276937639,variable=\t]({1.*5.843140397479143*cos(\t r)+0.*5.843140397479143*sin(\t r)},{0.*5.843140397479143*cos(\t r)+1.*5.843140397479143*sin(\t r)});
\draw [->,line width=1.pt] (-4.833333333333334,-0.24482956899224875) -- (-3.942422978811363,-0.1690362459948329);
\draw [->,line width=1.pt] (0.86,3.01) -- (1.76,3.01);
\draw [line width=0.4pt,color=black]
(-7.286911510639987,0.015399653820113275)-- (-3.71286396088189,0.);
\draw [->,line width=1.pt] (-5.436576752454775,0.006746550904390407) -- (-4.557656752454775,0.006746550904390407);
\draw [line width=0.4pt,color=black]
(-1.002187357889388,3.0002953070174363)-- (2.5954233563654627,2.988006777698276);
\draw [dotted,line width=0.4pt,color=black]
(-0.9836469719311594,0.015236956886861108)-- (2.5777373256253635,0.);
\draw [->,line width=1.pt] (0.8733333333333335,-0.003333333333334837) -- (1.7705179204134434,0.006746550904390944);
\draw [shift={(0.7671159797898673,-6.987831678012875)},line width=0.4pt,color=black]
plot[domain=1.317261520413934:1.8157744245219547,variable=\t]({1.*7.218596900938256*cos(\t r)+0.*7.218596900938256*sin(\t r)},{0.*7.218596900938256*cos(\t r)+1.*7.218596900938256*sin(\t r)});
\draw [->,line width=1.pt] (0.4033183889750277,0.2215921915589991) -- (1.3017637953488435,0.2899521681309197);
\begin{scriptsize}
\draw [fill=gray] (-5.44927,3.03) circle (2pt);
\draw [fill=gray] (-4.833333333333334,-0.24482956899224875) circle (2pt);
\draw [fill=gray] (0.86,3.01) circle (2pt);
\draw [fill=gray] (0.8733333333333335,-0.003333333333334837) circle (2pt);
\draw [fill=gray] (-5.436576752454775,0.006746550904390407) circle (2pt);
\draw [fill=gray] (0.4033183889750277,0.2215921915589991) circle (2pt);
\draw [->,color=black] (-3.2,1.5) -- (-1.3,1.5);
\node at (-5.4,3.35) {$(v,y_0)$};
\node at (0.8,3.35) {$\hat F(v,y_0) = (v,y_0)$};
\node at (0.9,-0.3) {$(v,y)$};
\node at (-5.4,0.3) {$(v,y)$};
\node at (0.35,0.55) {$\hat{F}(v,y)$};
\node at (-4.9,-0.525) {$\hat{F}^{-1}(v,y)$};
\end{scriptsize}
\node at (-7.8,3) {$g_\mathbb{D}$};
\node at (-7.8,0) {$g_\mathbb{D}$};
\node at (3.1,3) {$g_{y_0}$};
\node at (3.1,0) {$g_y$};
\node at (-2.25,1.8) {$\hat F$};
\end{tikzpicture}

 \caption{The map $\hat F$ transforming the metrics $g_y$ into the metric $g_\mathbb{D}$}  \label{figura}
\end{figure}
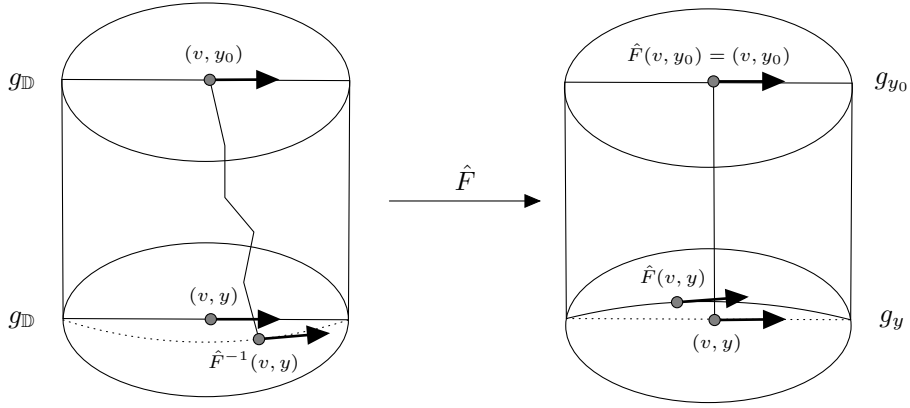


\begin{proof} Let $\varepsilon_0>0$ be such that the map $\hat \tau$ is injective on any ball of radius $\varepsilon_0$, and let $\varepsilon \leq \varepsilon_0/3$. We will also assume that the diameter of $T_0$ is smaller than $\varepsilon$. 
As we said in Remark~\ref{remark:tangent_metric_riemannian}, the leaves $\D\times\{y\}$ of the horizontal foliation $\mathcal{H}$  carry complete metrics $g_y$ of constant curvature $-1$.

Without loss of generality, we will assume that for some (fixed) $y_0\in T_0$ the metric $g_{y_0}=g_{\D}$ is the complete hyperbolic metric corresponding to the standard conformal class $|dz|$. For each $y\in T_0$, the metric $g_y$ corresponds to the class $|dz+\mu_yd\bar z|$. The Beltrami coefficient $\mu_y$ is a smooth complex-valued function on $\D$. When $y$ varied in $T_0$, the metric $g_y$ varies continously in the smooth topology, and so does $\mu_y$. Furthermore, since $M$ is compact and the leaves of the Riemannian foliation are locally at constant distance, $\mu_y$ varies continously in the uniform topology. Taking a smaller $T_0$ if necessary, we may assume that there is a $k<1$ such that $||\mu||_\infty<k$ for any $y\in T_0$.

Then by the Ahlfors-Bers theorem (see \cite{Ahlfors-Bers}), for any $y$, there is a unique $k$-quasiconformal diffeomorphism $f_y:\D\to \D$ that solves the Beltrami equation 
$$
\frac{(f_y)_z}{(f_y)_{\bar z}}=\mu_y
$$
and which fixes the three points $ -1,\, i, \, 1$. Recall that any quasiconformal map of $\D$ extends to a homeomorphism of $\bar D$. Moreover, $f_y$ depends continuously on $y$ in the uniform topology, as explained in \cite[Section 5]{Ahlfors-Bers}. Now, $f_y$ pulls back the conformal class $|dz+\mu_y d\bar z|$ to $|dz|$, and hence the hyperbolic metric $g_y$ to $g_{\D}$. This means that when $T_0$ is small the diffeomorphism $f_y$ is an isometry between $(\D, g_{\D})$ and $(\D, g_{y})$ which is close to the identity, uniformly.

Now, we define a homeomorphism 
$$
F : \D \times T_0 \to \D \times T_0
$$
by 
$$
F(x,y) = (f_y(x),y),
$$
which restricts to a smooth diffeomorphism from each horizontal leaf $\D \times \{y\}$ to itself. By construction, the metric $g_y$ on $\D \times \{y\}$ is now pulled-back to the hyperbolic metric $g_{\D}$. However, the vertical leaf $\{x\} \times T_0$ is not pulled-back to a manifold orthogonal to the leaves of $\mathcal{H}$ since its preimage by $F$ is equal to the subset $\{(f^{-1}_y(x),y) : y \in T_0\}$ of $\D \times T_0$. Nevertheless, as each diffeomorphism $f_y$ is isotopic to the identity, there is a global isotopy $H : \D \times [0,1] \times T \to \D \times T_0$ such that $H(x,0,y) = (x,y)$ and $H(x,1,y) = (f_y^{-1}(x),y)$, and its 
restriction to $\{x\} \times [0,1] \times T_0 $ defines a homotopy between $\{x\} \times T_0$ and $\{(f_y^{-1}(x),y) : y \in T_0\}$. 
By continuity, there is $\eta >0$ such that if $T_0$ is restricted to the ball of radius $\eta/2$ centred at $y_0$, then the diffeomorphism $f_y$ becomes $\varepsilon$-close to the identity in $\rm{Diff}_0(\D;1,i,-1)$, that is, 
$$
\sup \{ d(f_y^{-1}(x),x) : y \in T_0 \} < \varepsilon
$$ 
and then 
$d(f_y^{-1}(x),x) < \varepsilon$ for all $(x,y) \in \D\times T_0$.
For the Sasakian metric on $T^1\D \times T_0$ which is induced by the Riemannian metric $\tau^\ast g$ on $\D \times T_0$, we have a homeomorphism 
$$
\begin{array}{rcl}
\hat F : T^1\D \times T_0 &\to& T^1\D \times T_0\\
\hat F(v,y) &=& ((f_y)_*(v),y)
\end{array}
$$
which, in the leaf direction, pulls back the Sasakian metric corresponding to each $g_y$ to the one corresponding to $g_{\D}$ (see Figure~\ref{figura}).

Notice that since $f_y:(\D,g_{y_0})\to (\D,g_y)$ is an isometry, it sends geodesics to geodesics. This implies that for all $(v,y_0)\in T^1\D\times\{y_0\}$ and $y\in T_0$ and for all $t\in\R$
$$d(g_t(v,y_0),g_t(\hat F(v,y)) <2\varepsilon.$$
This means that the geodesic directed by $\hat F(v,y)$ {\em shadows} the geodesic directed by $(v,y_0)$, in the sense that it follows it at a small uniform distance, for the distance $d$ coming from the bundle-like metric in the normal tube.

Therefore, we have that 

$$W^{wu}_\varepsilon (\hat u)\subset \hat\tau( \hat F (W^{wu}_{\varepsilon} (u) \times T_0))\subset W^{wu}_{ 2\varepsilon}(\hat u),$$ 
as requested.
\end{proof}

Proposition~\ref{prop:Riemannian} is the analogue of Proposition~\ref{prop:homogeneus} for Riemannian foliations. Therefore, as stated in Remark~\ref{remark:proof_of_main_theorem_works_in_general_case}, the proof of Proposition~\ref{thm:homogeneous} applies to this case. 
 Let us briefly sketch the minor changes that need be made to this proof:

\begin{proof}[Proof of Theorem~\ref{thm:uniqueergodicity}]
We have to verify that we are in the hypotheses of Coudène's theorem.  As 
 in the proof of Proposition~\ref{thm:homogeneous}, we know 
 that the horocycle flow $h_s$ is transitive. We have to prove that the volume measure $\mu$ that comes from the metric $g$ is absolutely continuous with respect to $W^{wu}$. Again, the condition is semi-local, so rather than doing it in $\hat M$ we will do it in a normal tube $(\D\times T_0,\mathcal{H})$. Let $\hat u=(u,y_0)$ be a vector tangent to a leaf in this tube. According to Proposition~\ref{prop:Riemannian}, modulo a homeomorphism $F$ that is smooth along the horizontal leaves, the tube $\D\times T_0$ can be written as $F^{-1}(\D\times T_0)=\H\times T_0$, where the $\H$ factor is a plane with a fixed hyperbolic metric --the metric $g_{y_0}$ on the disk $\D$-- and the transversal $T_0$ is now a topological manifold with no metric structure. There is, nevertheless, a distance $d$ on $\H\times T_0$ that comes from the distance in the tube $\D\times T_0$.
 The unit tangent bundle to the horizontal foliation can be written as
 $T^1\H\times T_0,$
 where any two horizontal factors are isometric via the trivial map which is the identity on the first coordinate. 
 
 We consider $s\in (-1,1)$ and $\hat v=(v,y)\in W^{wu}_\varepsilon(\hat u)$, and we have to verify that
 $$W^{wu}_\varepsilon(h_s(u))\cap h_{(-2,3)}(\hat v)$$
 has exactly one point. Namely, that there are unique $l\in (-2,3)$ and $w\in W^{wu}_\varepsilon(h_s(u))$ such that $w=h_l(v)$. We are intersecting some open subsets of a weakly unstable and a stable manifold for the geodesic flow in $T^1\H\times\{y\}$, so there is at most one intersection point. Again, if the $\delta$ appearing in Definition~\ref{definition:absolute_continuity} is small enough, we can guarantee that this point exists.
 
 The condition on the disintegration of the measure $\mu$ is verified exactly as in the case of homogeneous foliations. The transverse component of the measure plays no role, and the tangential component, being invariant under both the geodesic and stable horocycle flow, has the desired property.
\end{proof}

\end{document}